\newif\ifSC
\ifSC
\documentclass[onecolumn,draftclsnofoot,12pt]{IEEEtran}
\else
\documentclass[10pt,twocolumn]{IEEEtran}
\fi
\usepackage{cite}
\usepackage{amsmath,amssymb,amsfonts,amsthm}

\usepackage{algorithmic}
\usepackage{graphicx}
\usepackage{textcomp}
\usepackage{color}
\usepackage{varioref}
\usepackage{etoolbox}
\usepackage{xcolor}
\usepackage{bbm}
\usepackage{lipsum}
\usepackage[normalem]{ulem}
\usepackage{mathtools}
\usepackage{subfigure}
\usepackage{pstricks}
\usepackage{xcolor}

\def\home{\hbox{\kern3pt \vbox to13pt{}%
		\pdfliteral{q 0 0 m 0 5 l 5 10 l 10 5 l 10 0 l 7 0 l 7 5 l 3 5 l 3 0 l f
			1 j 1 J -2 5 m 5 12 l 12 5 l S Q }%
		\kern 13pt}}

\newtheorem{theorem}{Theorem}
\newtheorem{lemma}{Lemma}

\newtheorem{assumption}{}
\newtheorem{feature}{}

  \usepackage{mathtools} 
  
\def\BibTeX{{\rm B\kern-.05em{\sc i\kern-.025em b}\kern-.08em
		T\kern-.1667em\lower.7ex\hbox{E}\kern-.125emX}}
\newcounter{relctr} 
\everydisplay\expandafter{\the\everydisplay\setcounter{relctr}{0}} 

\newcommand\numeq[1]%
{\stackrel{\scriptscriptstyle(\mkern-1.5mu#1\mkern-1.5mu)}{=}}
\AtBeginDocument{}

\newcommand{\p}{\mathbf{p}}
\newcommand{\x}{\mathbf{x}}

\newcommand{\y}{\mathbf{y}}

\graphicspath{{figs/}}

\newcommand{\e}{\mathbf{e}}
  
\begin{document}	
\title{\small\textcolor{red}{THIS IS A DRAFT VERSION AND IS SUBJECT TO REVISION.}\\
	\huge Distributed Optimisation With Communication Delays}
\author{\IEEEauthorblockN{ Shuubham Ojha, Ketan Rajawat}
	\thanks{The authors are with Indian Institute of Technology Kanpur, Kanpur (India) 208016 (Email: shuubham@iitk.ac.in, ketan@iitk.ac.in).}
}

\maketitle

\begin{abstract}
This paper discusses distributed optimization over a directed graph. We begin with some well known algorithms which achieve consensus among agents including Fast Row-stochastic-Optimization with uncoordinated STep-sizes (FROST) \cite{xin2019frost}, which possesses linear convergence to the optimum. However FROST works only over fixed topology of underlying network. Moreover the updates proposed therein require perfectly synchronized communication among nodes. Hence communication delays among nodes, which are inevitable due to processing delays and/ or channel impairments, preclude the possibility of implementing FROST in the real world. In this paper we utlize a co-operative control strategy which makes convergence to optimum impervious to communication delays.
\end{abstract}

\section{Introduction}
Distributed optimization has been a subject of research for its numerous advantage over a centralized approach. A decentralized approach to optimization involves independent optimization of local functions at each node, unlike centralized optimization, which helps in parallel processing of data as well as in privacy preservation because each node is aware of it's local function only. The goal of such a decentralized optimization routine is to enforce consensus among nodes at the minimia of a global function, which is generally a weighted sum of the local functions. As in a classical consensus problem, each node acquires the local estimate of the optimum of the global function from it's neighbouring nodes and subsequently updates it's own estimate. The well known Fast Row-stochastic-Optimization with uncoordinated STep-sizes (FROST) algorithm \cite{xin2019frost} belongs to the lineage of distributed optimization algorithms. It allows nodes to converge to the optimizer of the global function under certain assumptions on nature of local functions and topology of the network over which communication among nodes takes place. The presence of adversarial nodes in a network along with non-adversarial ones botches up the attainment of the global optimum. \cite{sundaram2018distributed} gives a certain guarantee that the regular nodes can attain consensus under malicious behaviour though the consensus value is still far away from the optimum of the global objective function. In this paper we restrict our attention to adversary free networks.\par

The broadcast model of communication, which is utilised in most modern networks involves iterative telecast of state estimates by message sending nodes, which are subsequently received by neighbouring nodes whose identity is hidden from the sender. Since this model of communication does not require the sender to know the number or identity of the receivers it precludes the existence of network topologies whose adjacency matrix is doubly stochastic or column stochastic. The broadcast model of communication reinforces the relevance of the FROST like algorithms, which can be applied when nodes are commnicating over directed graphs with row-stochastic adjacency matrix to ensure linear convergence to the global optimum, provided the local functions are strongly convex and smooth. The utility of FROST is limited by the fact that convergence to the global optimum in presence of communication delays as well as abrupt changes to topology is not guaranteed. Since the gradient tracking iterate enables FROST to converge linearly, it is imperative that we adapt this iterate to the asynchronous case should our algorithm be required to inherit the convergence properties of FROST. In \cite{qu2014cooperative} the authors introduce a control law which dynamically estimates the first left eigenvector of the network adjacency matrix for changing topologies and communication delay between nodes. This estimated eigenvector can be utilized by a First Order Dynamic Average Consensus (FODAC) Algorithm \cite{zhu2010discrete} to perform average tracking in a row-stochastic setting. In the same spirit \cite{du2020accurate} introduces a delay robust average tracking mechanism, which allows tracing the average of arbitrary dynamic quantities in presence of delays. Our idea is to utlise this mechanism to track the average of gradients locally and utlize this average in a gradient descent equation to converge at the global optimum. \par

In this paper, we relax the requirement of synchronised communication while employing FROST by introducing a distributed co-operative control law with an associated gain variable. As before, we assume that the nodes are unaware of each others local function. The main objective of this paper is to reach global optimum of the sum of these local functions in the presence of communication delays under a suitable choice for the gain parameter. The paper is organized in the following manner. Section II outlines the notation that have been consistently used throughout the paper. In section III we draw a basic idea and  terminology for the distributed network. In section IV we revisit some related work and existing algorithms in distributed optimization and proceed to formalize the notion of asynchronous operation. Section V delineates the algorithm development and the intuition behind it. Section VI gives the convergence analysis while section VII provides simulation results to substantiate the theoretical results. Section VIII concludes the work with possible directions for future work.

\section{Notation}
Scalars are represented by small letters, vectors by small boldface letters, sets by calligraphic letters and matrices are denoted by capital boldface letters. The notation  $[\mathbf{u}]_{i}$ denotes the $i^{th}$ element of this vector $\mathbf{u}$. The index $t$ is used for the time or iteration index. The inner product between vectors $\mathbf{a}$ and $\mathbf{b}$ is denoted by $\langle\mathbf{a},\mathbf{b}\rangle$. The kroenecker product of matrices $\mathbf{A}$ and $\mathbf{B}$ is denoted by $\mathbf{A} \otimes \mathbf{B}$. The notation $\nabla$ denotes the gradient and $\|.\|$ denotes the Euclidean Norm. We use the notation $\mathbf{1}_{N}$ to denote the $N$ dimensional row vector consisting of all ones. The symbol $\dot{\x}(t) = \x(t+1) - \x(t)$ signifies the discrete time derivative of $\x$. 
 
\section{Terminology for a Distributed Network}
A graph is defined by $\mathcal{G} = (\mathcal{V}, \mathcal{E})$,  which contains a set of vertices (or nodes) $\mathcal{V} = \{v_i\}_{1\leq i \leq N}$, and a set of edges represented by $\mathcal{E} \subset \mathcal{V} \times \mathcal{V} = \{v_i,v_j\}_{1\leq i,j\leq N, i \neq j}$. If for each $i$,$j$ $\in$ $\{1,\hdots N\}$, $(v_i,v_j)\in \mathcal{E}$ $\implies$ $(v_j,v_i)\in \mathcal{E}$, graph $\mathcal{G}$ is called undirected otherwise it is called a directed graph. The set of in-neighbours for a node $v_i \in \mathcal{V}$ is defined by $\mathcal{N}_i^{-} \triangleq \{v_j \in \mathcal{V} |(v_j,v_i)\in \mathcal{E}\}$, which is the set of agents that can send information to $v_i$, while the set of out-neighbours is defined as $\mathcal{N}_i^{+} \triangleq \{v_j \in \mathcal{V} |(v_i,v_j)\in \mathcal{E}\}$. The in-degree and out-degree is the size or cardinality of respective set represented as $d_i^{-} \triangleq |\mathcal{N}_i^{-}|$ and $d_i^{+} \triangleq |\mathcal{N}_i^{+}|$. Clearly for an undirected graph $ d_i \triangleq d_i^{-} = d_i^{+}$. We associate an adjacency matrix $\mathbf{A} = \{a_{ij}\}_{1\leq i,j\leq N}$ with graph $\mathcal{G}$ such that $a_{ij} > 0$ whenever $(v_j,v_i)\in \mathcal{E}$ and denotes the weight of this edge, otherwise $a_{ij} = 0$. We define the degree matrix as a diagonal matrix $\mathbf{D}$ = $\{d_{ij}\}_{1\leq i,j\leq N}$ with $d_{ii} \triangleq \sum\limits_{j=1}^{n}$ $a_{ij}$. The Laplacian $\mathbf{L}$ is defined as $\mathbf{L}$ = $\mathbf{D}$ - $\mathbf{A}$. The normalized Laplacian is defined as $\mathbf{L'}$ = $\mathbf{D}^{-1}\mathbf{L}$. A graph is said to be strongly connected if there is a directed path from each node to another node in the graph. Throughout this paper we assume that matrix $\mathbf{A}$ is row stochastic and denote by $\mathbf{u}$ the left eigenvector of $\mathbf{A}$ corresponding to eigenvalue 1.

\section{Problem Formulation}
In a distributed network each node has its own local function $f_i$ such that $f_i:\mathbb{R}^{n}\rightarrow\mathbb{R}$ is convex with bounded subgradients and is only available to node $v_i$. Our goal is to solve a global optimization problem which we formulate as follows,
\begin{align*}
\underset{\mathbf{x}\in\mathbb{R}^{n}}{\min}F(\mathbf{x}) \triangleq \frac{1}{N}\sum\limits_{i=1}^{N}f_{i}(\mathbf{x}). \tag{P1}
\end{align*}
Distributed optimization envisages a distributed algorithm which enables each agent to converge to the global solution of Problem P1 by exchanging information with nearby agents over a directed graph. We formalize the set of assumptions which are standard in the literature for optimization of smooth convex functions and are essential for ensuring convergence.
\begin{assumption}\label{asenv1}
	The underlying graph is directed and strongly connected.
\end{assumption}
\begin{assumption}\label{asenv2}
	Each local function is convex with bounded sub-gradient.
\end{assumption}
\begin{assumption}\label{asenv3}
	Each local function $f_i$ is smooth with constant $l_i$ is strongly convex with constant $\sigma_{i}$.
\end{assumption}
\begin{assumption}\label{asenv4}
	Each node $v_i$ knows the number of out-neighbours it possesses.
\end{assumption}
Distributed Gradient Descent (DGD) is a popular algorithm for solving (P1) where each agent $v_i$ maintains a local estimate $\x_i(t+1)$ and implements the following update iteration \cite{xi2017distributed}:
\begin{align}
\x_i(t+1) = \sum\limits_{j=1}^{N} a_{ij}\x_j(t) - \alpha_t\nabla f_i(\x_i(t)).
\end{align}
where $\mathbf{A} = \{a_{ij}\}_{1\leq i,j \leq N}$ is doubly stochastic and $\alpha_t$ is a diminishing step size satisfying $\sum\limits_{t=0}^{\infty}\alpha_t = \infty$ and $\sum\limits_{t=0}^{\infty}\alpha_t^2 < \infty$ and $\nabla f_i(x_i(t))$ is the local gradient calculated at each node.
To accelerate the convergence rate of DGD a method based on gradient tracking is proposed \cite{xu2015augmented} which involves an additional variable $\mathbf {y}_i(t)$, which tracks the average of gradient of nodal functions $f_{i}$. The new iteration is:
\begin{align}
\mathbf{x}_{i}(t+1)&=\sum\limits_{j=1}^{N}a_{ij}\mathbf{x}_{j}(t)-\alpha\mathbf{y}_{i}(t)\\
\mathbf{y}_{i}(t+1)&=\sum\limits_{j=1}^{N}a_{ij}\mathbf{y}_{j}(t)+\nabla f_{i}\left(\mathbf{x}_{i}(t+1)\right)-\nabla f_{i}\left(\mathbf{x}_{i}(t)\right).
\end{align}
Here update (2) is a descent equation in which $\x_i(t)$ is the estimate of $\x^{*}$ by node $i$ at iteration $t$ and $\mathbf {y}_i(t)$ in update (3) tracks the average of local gradients $\frac {1}{N}\sum\limits_{i=1}^{N}\nabla f_{i}\left (\mathbf {x}_i(t)\right)$, provided $\y_{i}(0) = \nabla f_i(\x_i(0))$. It is shown in \cite{xi2015directed} that  $\mathbf {x}_i(t)$  converges linearly to $\x^{*}$ when $\mathbf{A}$ is doubly stochastic, $f_{i}'s$ are strongly convex and $\alpha_{t}$ is a sufficiently small and constant step size with $\alpha_t = \alpha$.  \par
When $\mathbf{A}$ is column-stochastic, \cite{xi2017add} introduces the ADD-OPT algorithm, whose convergence is assured under assumptions \ref{asenv1}-\ref{asenv4} and which involves the iterates:
\begin{align} 
&\mathbf {x}_{i}(t+1)=\sum _{j\in \mathcal {N}_i^{{\rm{-}}}}a_{ij}\mathbf {x}_{j}(t)-\alpha \mathbf {w}_{i}(t) \\ 
&\mathsf {y}_{i}(t+1)=\sum _{j\in \mathcal {N}_i^{{\rm{-}}}}a_{ij}\mathsf {y}_{j}(t)\\ 
&\mathbf {z}_{i}(t+1)=\frac{\mathbf {x}_{i}(t)}{\mathsf {y}_{i}(t)}\\ 
&\mathbf {w}_{i}(t+1)=\sum _{j\in \mathcal {N}_i^{{\rm{-}}}}a_{ij}\mathbf {w}_{j}(t)+\nabla f_i(\mathbf {z}_{i}(t+1))-\nabla f_i(\mathbf {z}_{i}(t))
\end{align}
where $\alpha$'s is a constant step size chosen mutually by each agent. $\y_{i}(t)$ is initialised as $\y_{i}(0) = 1$. The update (5) tracks the first left eigenvector of $\mathbf{A}$. Update (7), which is analogous to (3) in the column stochastic setting, tracks the average of local gradients $\frac {1}{N}\sum\limits_{i=1}^{N}\nabla f_{i}\left (\mathbf {z}_i(t)\right)$, provided $\mathbf{w}_{i}(0) = \nabla f_{i}(\mathbf{z}_{i}(0))$. Update (4) is essentially a gradient descent equation where the descent direction is $\mathbf{w}_i(t)$ instead of $\nabla f_{i}\left(\mathbf{x}_i(t)\right)$ and update (6) normalizes the node estimates $\x_{i}(t)$ which is essential for producing consensus. \par  
An additional algorithm based on the gradient tracking called FROST algorithm is introduced in \cite{xin2019frost} for the case when $\mathbf{A}$ is row-stochastic. It's convergence is ensured under assumptions \ref{asenv1}-\ref{asenv3}. The updates involved are given as follows:
\begin{align}
&\mathbf{y}_i(t+1)=\sum\limits_{j=1}^{N}a_{ij}\mathbf{y}_j(t)\\
&\mathbf{x}_i(t+1)=\sum\limits_{j=1}^{N}a_{ij}\mathbf{x}_j(t)-\alpha_t\mathbf{z}_j(t)\\
&\mathbf{z}_i(t+1)=\sum\limits_{j=1}^{N}a_{ij}\mathbf{z}_i(t)+\frac{\nabla f_{i}\left(\mathbf{x}_i(t+1)\right)}{\left[\mathbf{y}_i(t+1)\right]_{i}}-\frac{\nabla f_{i}\left({\vphantom{\mathbf{y}_i(t+1)}} \mathbf{x}_i(t)\right)}{\left[{\vphantom{\mathbf{y}_i(t+1)}} \mathbf{y}_i(t)\right]_{i}}
\end{align}
where $\alpha_t$'s are the uncoordinated step-sizes locally chosen at each agent. $\y_{i}(t)$ is initialised as $\y_{i}(0) = \e_{i}$, the $i^{th}$ basis vector of $\mathbb{R}^{N}$. The update (8) tracks the first left eigenvector of $\mathbf{A}$. Update (10), which is analogous to (3) in the row stochastic setting, tracks the average of local gradients $\frac {1}{N}\sum\limits_{i=1}^{N}\nabla f_{i}\left (\mathbf {x}_i(t)\right)$, provided $\mathbf{z}_{i}(0) = \nabla f_{i}(\x_{i}(0))$ and update (9) is essentially a gradient descent equation where the descent direction is $\mathbf{z}_i(t)$ instead of $\nabla f_{i}\left(\mathbf{x}_i(t)\right)$. \par 
Within the distributed optimization framework considered here, the focus is on distributed algorithms that can be applied to optimize a sum of smooth and strongly convex functions by message passing among nodes connected by a directed graph whose  topology is defined by row-stochastic and column-stochastic adjacency matrices. Each of these algorithms, however, require perfectly synchronized communication among nodes so that updates at time instant $t+1$ utilize node estimates at time instant $t$. Our goal in this paper is to develop an asynchronous analog of the FROST algorithm, delineated above, that would allow some of the nodes in the network to  “lag behind” in the event they are unable to supply their out-neighbours with the latest copy of their local variables due to large communication bottleneck and/or processing delays. Our asynchronous algorithm utilizes a dynamic average tracking mechanism which is robust to bounded communication delays. The key feature of the proposed algorithm is the adjustable gain parameter $\kappa'$, which is a function of the expected delay on a given communication link. We begin with enlisting the desirable features of an algorithm that seeks to solve (P1) in presence of delays. Specifically, it is required that any such algorithm meets the following requirements.
\begin{feature}
	The algorithm should allow nodes to “lag behind” due to poor channel condition and/or processing delays and supply their out-neighbours with the older copies of their local variables. 
\end{feature}
\begin{feature}
	The algorithm should allow a distributed implementation, that is, without requiring a central server that communicates with every node.
\end{feature}

The gradient tracking technique that has been employed to develop 
decentralized algorithms to track the average of the gradients \cite{nedic2017achieving}, \cite{tian2020achieving} enjoy linear convergence even under time-varying communication topologies. The next section addresses the challenge of proposing a novel gradient tracking algorithm for distributed optimization that is robust to delays in information exchange. Equipping our algorithm with the gradient tracking technique will help it in inheriting the convergence properties of FROST.  

\section{Algorithm Development}.
This section details the delay robust directed distributed gradient descent algorithm that incorporates the features (F1)-(F2) in its design. We begin with motivating the design of our algorithm by rewriting iterate (10) from FROST algorithm in matrix form as follows,  
\begin{align}
 \mathbf{z}(t+1)=\bar{\mathbf{A}}\mathbf{z}(t)+\widetilde{Y}_{t+1}^{-1} \nabla\mathbf{f}\left(\mathbf{x}(t+1)\right)-\widetilde{Y}_{t}^{-1} \nabla\mathbf{f}\left(\mathbf{x}(t)\right)
\end{align}
where $\bar{\mathbf{A}} = \mathbf{A}\otimes\mathbf{I}_{n}$, $\underline{Y}_{t}=\left[\mathbf{y}_{1}(t),\cdots,\mathbf{y}_{N}(t)\right]^{\top}, Y_{t} = \underline{Y}_{t} \otimes \mathbf{I}_{n}, \widetilde{Y}_{t}=\text{diag}\left(Y_{t}\right)$ with $\underline {Y}_{0}=\mathbf{I}_{N}$, $\mathbf{z}_{0} = \nabla{\mathbf{f}(\x(0))}$ and $\nabla\mathbf{f}\left(\mathbf{x}(t)\right)$, $\mathbf{z}(t)$ collects the variables $\nabla f_{i}\left(\mathbf{x}_{i}(t)\right)$ and $\mathbf{z}_{i}(t)$ respectively in $\mathbb{R}^{Nn}$. \par
Multiplying (11) by $\bar{\mathbf{u}}=\mathbf{u}\otimes\mathbf{I}_{n}$ and recalling that $\bar{\mathbf{u}}(\mathbf{A}\otimes\mathbf{I}_{n}) = \bar{\mathbf{u}}$ we get,
\begin{align}
 \bar{\mathbf{u}}\mathbf{z}(t+1)=\bar{\mathbf{u}}\mathbf{z}(t)+\bar{\mathbf{u}}\widetilde{Y}_{t+1}^{-1}\nabla\mathbf{f}(\mathbf{x}(t+1))-\bar{\mathbf{u}}\widetilde{Y}_{t}^{-1}\nabla\mathbf{f}(\mathbf{x}(t))
\end{align} 
Repeating (12) for all values of $t$ and adding we get,
\begin{align}
 \bar{\mathbf{u}}\mathbf{z}(t+1)=\bar{\mathbf{u}}\mathbf{z}(0)+\bar{\mathbf{u}}\widetilde{Y}_{t+1}^{-1}\nabla\mathbf{f}(\mathbf{x}(t+1))-\bar{\mathbf{u}}\widetilde{Y}_{0}^{-1}\nabla\mathbf{f}(\mathbf{x}(0))
\end{align}
Given the intial conditions $\widetilde{Y}_{0}=\mathbf{I}_{Nn}$ and $\mathbf{z}(0)=\nabla\mathbf{f}(\x(0))$, we have,
\begin{align}
\bar{\mathbf{u}}\mathbf{z}(t+1)=\bar{\mathbf{u}}\widetilde{Y}_{t+1}^{-1}\nabla\mathbf{f}(\mathbf{x}(t+1))
\end{align}
Equation (14) demonstrates how $\mathbf{z}(t)$ tracks the average of local gradients $\nabla{f}_{i}(\x_{i}(t))$ weighted by components of $\mathbf{u}$ in the absence of delays. We now take a detour to study average tracking in presence of delays before returning to fuse the delay robust average tracker with gradient descent to propose our algorithm. Section V-A describes the synchronous variant of an average consensus algorithm, which allows every node to converge at the average of initial values of arbitrarily changing local quantities. Next, Section V-B details the asynchronous variant of the average consensus algorithm, which allows every node to converge at the average of arbitrarily changing local quantities in presence of time varying bounded delays. Finally, our proposed algorithm which is a revised version of FROST that is impervious to bounded communication delays is presented in section V-C. 
\subsection{Average Tracking Without Delays}
Average tracking without delays in a multi agent setup is well known in literature. An algorithm to achieve consensus at the average of initial values of  dynamic quantities in absence of delays can be expressed as \cite{atay2010consensus}:
\begin{align} \dot {\x}_{i}(t)=-\kappa '\sum _{(j,i)\in \mathcal{E} } a_{ij}\left [{\x_{i}(t)-\x_{j}(t)}\right]\end{align}
where $\x_i(t)$ represents the real-time state of the $i^{th}$ agent at time instant $t$, $x_{i}^{0} = \x_{i}(0)$ represents the initial value of $\x_i(t)$ at $t = 0$. In \cite{atay2010consensus} authors show that all states in (11) reach the average consensus asymptotically, i.e., let $x_{i}^{*}=\displaystyle\lim _{t \to \infty }x_{i}(t)$, then $c_{x}=x^{*}_{i}=\dfrac {1}{N} \displaystyle{\sum _{i\in \mathcal{V} }} x_{i}^{0}$ and $c_{x}$ represents the consensus equilibrium of $\x_i(t)$. 

\subsection{Dynamic Average Tracking With Delays}
In the previous subsection we demonstrated consensus assuming delay free communication. In this subsection, we focus on the realistic case of tracking averages of time varying quantities in presence of delayed communication. Let $\mathbf{b}_{i}(t)$ be input dynamic state at node $i$ whose average is desired to be tracked and $\tau(t)$ be time varying delay with which each node receives information from it's in-neighbours. We introduce $\mathbf{r}_{i}(t)$ and $s_{i}(t)$ as auxilliary variables which will assist us in the average tracking process. $\kappa$ is a tunable scalar gain parameter which is a function of expected delay on a particular communication link and $\kappa'$ is $\kappa$ normalised by $d_{ii}$. Let $g_{i}(t)$ denote a square wave local to node $i$ such that all the $g_{i}(t)'s$ are perfectly synchronized. The time period of $g_{i}(t)$ is $T_g$ for each $i$ and it is given by,
\begin{align*}
g_{i}(t) = 0.5\bigg(sgn\bigg(\sin \frac{2\pi t}{T_g}\bigg) + 1\bigg)
\end{align*} 
We begin by reproducing the iterations from \cite{du2020accurate} as follows:
\begin{align} \label{12}
&\dot{{\mathbf{r}}}_i(t)= \dot{{\mathbf{b}}}_i(t)- \kappa' \sum\limits_{i,j \in \mathcal{E}} a_{ij}({\mathbf{r}}_i(t)-{\mathbf{r}}_j(t-\tau(t))) \tag{16a} \\
&\dot{{s}}_i(t)= \dot{{g}}_i(t) - \kappa'\sum\limits_{i,j \in \mathcal{E}} a_{ij}({s}_i(t)-{s}_j(t-\tau(t))) \tag{16b}
\end{align}  
where, $s_{i}(0)$ = $g_{i}(0) = 1$ and $\mathbf{r}_{i}(t) = 0$ for $t < 0$, $\mathbf{r}_{i}(0) = \mathbf{b}_{i}(0)$. Further it is assumed that $\tau(t)$ is sampled from a distribution with finite support for each $t$ and is therefore bounded by some $\tau_{max} > 0$. We denote $\mathbb{E}(\tau(t))$ by $\bar{\tau}$. Let $\dot{\mathbf{b}}$ = $[\Delta\mathbf{b}_{i}(0)]_{1\leq i \leq N}$ be a perturbation in $\mathbf{b}(0) = [\mathbf{b}_{i}(0)]_{1\leq i \leq N}$, when the system (16) is assumed to be at consensus equilibrium. If the initial equilibrium value is given by $\mathbf{r}_i$ = $c_{r}'$ and $s_i$ = $c_{s}'$ for each $i$, the new equilibrium states $c_r$ and $c_s$ are given by \cite{du2020accurate}: 
\begin{align*}
c_{r} = c_{r}' + \frac{(\mathbf{u} \otimes \mathbf{I}_{n})\dot{\mathbf{b}}}{1+\kappa\bar{\tau}} \tag{17a}\\
c_{s} = c_{s}' + \frac{\Delta g}{1+\kappa\bar{\tau}} \tag{17b}
\end{align*}
where, $\mathbf{u}$ is the First Left Eigenvector (FLE) of $\mathbf{A}$, see \cite{atay2010consensus}. Since $|\Delta g|$ = 1, initialising $c_{r}'$ = $\frac{\mathbf{u}\mathbf{b}(0)}{1+\kappa\bar{\tau}}$ and ensuring that changes in $\mathbf{b}$ coincide with rising and falling edges of $g$ we have that,
\begin{align*}
c_{p} = \frac{c_{r}}{|c_s - c_{s}'|} \tag{17c}
\end{align*}
tracks the weighted average of components of $\mathbf{b}(0) + \dot{\mathbf{b}}$ where components are weighed by elements of $\mathbf{u}$. 


\subsection{Proposed Algorithm}
We now utilise the average tracker from (16) in the gradient descent equation to reach the global optimum. This leads to updates (18a-e):
\\
\\
\textit{Eigenvalue Update}: 
\begin{align}
&{\e}_{i}\bigg(t\frac{T_g}{2}\bigg) = \sum\limits_{j=1}^{N}a_{ij}{\e}_{j}\bigg(t\frac{T_g}{2}-\tau\bigg(t\frac{T_g}{2}\bigg)\bigg) \tag{18a}
\end{align}  
Update (18a) tracks the eigenvalue of the adjacency matrix $\mathbf{A}$ under the initialisation $\e_{j}(t) = \e_{j}$ for $t\leq 0$, where $\e_j$ is the $j^{th}$ basis vector of $\mathbb{R}^{N}$. The selected value of $T_g$ enables ${\e}_{j}\big(t\frac{T_g}{2}-\tau\big(t\frac{T_g}{2}\big)\big) = {\e}_{j}\big((t-1)\frac{T_g}{2}\big)$. \\
\\
\textit{Dynamic Average Tracking}:
\begin{align}
 \dot{{\mathbf{r}}}_i(t) = \hspace{97mm} \nonumber \\ \bigg(\frac{\nabla{f}_{i}({\x}_{i}(t))}{[\e_{i}(t)]_{i}} - \frac{\nabla{f}_{i}\big({\x}_{i}\big(t-\frac{T_g}{2}\big)\big)}{\big[\e_{i}\big(t-\frac{T_g}{2}\big)\big]_{i}}\bigg)   \sum\limits_{l_{0}=-\infty}^{\infty}\delta\bigg(t-l_{0}\frac{T_g}{2}\bigg) \nonumber \hspace{25mm} \\   - 
\kappa' \sum\limits_{(j,i) \in \mathcal{E}} a_{ij}({\mathbf{r}}_i(t) - {\mathbf{r}}_j(t-\tau(t)) \hspace{22mm} \tag{18b}
\end{align}
Update (18b) is analogous to update (16a) for the dynamic average tracking case. The impulse train in the above update forces the change in normalized gradient values to coincide with the rising and falling edges of $T_g$. Next, we turn our attention towards obtaining a handle on $\tau(t)$, the delay with which nodes transmit information to their out-neighbors. This is accomplished via the following iterate.    
\begin{align}
 \dot{s}_i(t)= \dot{g}_i(t)- \kappa' \sum\limits_{(j,i) \in \mathcal{E}} a_{ij}({s}_i(t) - {s}_j(t-\tau(t))) \tag{18c}
\end{align}  
Update (18c) is analogous to equation (16b) and tracks the value of $\frac{1}{1+\kappa\bar{\tau}}$ for each epoch of length $\frac{T_g}{2}$. In the spirit of (17c), we proceed by normalizing the equilibrium state value of (18b) with that of (18c). 
\begin{align} 
\p_i\bigg(t\frac{T_g}{2}\bigg)= \sum\limits_{l=1}^{t-1} \frac{\mathbf{c}_r\big(l\frac{T_g}{2}\big)}{|c_{s} - c_{s}'|} + (\mathbf{u} \otimes \mathbf{I}_n) \bigg[\frac{\nabla f_i(\x_i(0))}{[\e_{i}(0)]_{i}}\bigg]_{1\leq i \leq N}^{T}
\tag{18d}
\end{align}
Update (18d) normalizes the equilibrium state value of $\mathbf{r}_{i}(t)$, $\mathbf{c}_r$ with $c_s$, the equilibrium state value of $s_{i}(t)$, to cancel out the impact of delays on the average tracking process. This gives us the change in the weighted average of the gradients. The change is summed with the initial weighted average, given by the second term, to obtain the weighted average at time $\frac{tT_g}{2}$.  
\\
\\
Finally, update (18e) is the usual gradient descent step, where the direction of descent is $\p_i\big(t\frac{T_g}{2}\big)$, the weighted average of scaled gradients. The step size $\alpha_t$ is chosen mutually by all the agents at time instant $\frac{tT_g}{2}$. As before, our choice of parameter $T_g$ enables ${\x}_j\big(t\frac{T_g}{2}-\tau\big(t\frac{T_g}{2}\big)\big) = \x_{j}\big((t-1)\frac{T_g}{2}\big)$. \\
\\
\textit{Gradient Descent Step}:
\begin{align}
 {\x}_i\bigg(t\frac{T_g}{2}\bigg)=\sum\limits_{j=1}^{N}a_{ij}{\x}_j\bigg(t\frac{T_g}{2}-\tau\bigg(t\frac{T_g}{2}\bigg)\bigg)-\alpha_{t}{\p}_i\bigg(t\frac{T_g}{2}\bigg) \nonumber \\ \tag{18e}
\end{align}
In the next section, we will establish that the proposed updates (18a-e) converge for appropriate choice of step size $\alpha_{t}$ under assumptions \ref{asenv1}-\ref{asenv3}.

\section{Convergence Analysis}
This section provides the convergence result for our proposed algorithm, delineated in (18a-e). The final convergence result in Theorem 2 presents a suitable value of step-size $\alpha_{k}$ to be selected mutually by all agents at time instant $\frac{kT_g}{2}$ so as to ensure convergence. Before we present the final convergence proof, we begin by presenting a few intermediary lemmas on which our convergence result relies. \par
The first lemma analyzes the iterate (18d) and establishes it's convergence to a weighted average of local gradients normalized by FLE estimates.   
\begin{lemma}
 The iteration (18d) converges to the weighted average of gradients, $\frac{\nabla f_i(\x_i(t))}{[\e_{i}(t)]_{i}}$, as $t$ approaches time instants $\frac{mT_g}{2}$ for $m \in \mathbb{Z}$.
\end{lemma}
\begin{proof}
	Let $[\nabla f_i(k)]_{1\leq i \leq N}$ = $\big[\frac{\nabla f_1(\x_1(k))}{[\e_{1}(k)]_{1}} \hdots \frac{\nabla f_N(\x_N(k))}{[\e_{N}(k)]_{N}}\big]^{T}$. 
	From (13c) we have, 
	\begin{align*} 
	\frac{\mathbf{c}_r(l\frac{T_g}{2})}{|c_{s} - c_{s}'|} = (\mathbf{u} \otimes \mathbf{I}_n) \bigg[\nabla f_i\bigg(\frac{lT_g}{2}\bigg) - \nabla f_i\bigg(\frac{(l-1)T_g}{2}\bigg)\bigg]_{1\leq i \leq N}  
	\end{align*}
Substituting the above in (14d) we have, 
\begin{align*}
\p_i\bigg(t\frac{T_g}{2}\bigg) = (\mathbf{u} \otimes \mathbf{I}_n)\bigg[\nabla f_i\bigg(\frac{tT_g}{2}\bigg)\bigg]_{1\leq i \leq N}  
\end{align*}
which is the required result. 
\end{proof}

We have established the convergence of (18b), (18c) and (18d). In what follows we shall the establish the convergence of (18e) to the optimum $\x^{*}$ for each agent. To this end we shall first show in Lemma~\ref{lem4} that each agent converges to a consensus value. This is followed by Theorem~\ref{thm:errorinstrconv} and Theorem~\ref{Theorem3} which will give the final convergence result. \par
We take a detour to state and prove a property of smooth and strongly convex functions, known as the extension of co-coercivity property.
\begin{lemma}\label{lem:strongconv}
	Let, $\nabla F$ be Lipschitz contiuous with constant $L_{f} > 0$. If $F$ is strongly convex with parameter $\sigma_{f}$. Then, we have 
	\begin{align*}
	\langle \x - \y, \nabla F(\x) - \nabla F(\y) \rangle &  \geq \mu_1 \|\nabla F(\x) - \nabla F(\y)\|^2 \\
	& \hspace{0.5in} + \mu_2\|\x-\y\|^2,
	\end{align*}
	for all $\x,\y \in F$, with $\mu_1 = \frac{1}{\sigma_{f}+L_{f}}$ and $\mu_2 = \frac{\sigma_{f}L_{f}}{\sigma_{f}+L_{f}}$.
\end{lemma}
\begin{proof}
	Proof can be found in Appendix A.
\end{proof}

\begin{flushleft}
	Here onwards, $k$ denotes the time instant $\frac{kT_g}{2}.$ 
\end{flushleft}
The following lemma establishes the convergence of local estimate of global optimum by agent $i$, $\x_{i}(t)$ to a consensus value at a linear rate for any arbitrary initialisation of $\x(0)$. 

\begin{lemma} \label{lem4}
	Let, $\x(k)$ = $[\x_i(k)]_{1\leq i \leq N}^{T}$ be the sequence generated in (13e) and $\mathbf{\bar{u}} = \mathbf{u} \otimes \mathbf{I}_n $. Then, we have 
	\begin{align*}
	\|(\mathbf{1}_{N}^{T}\otimes\mathbf{\bar{u}})\x(k) - \x(k)\| = \|(\mathbf{A}\otimes \mathbf{I}_n-(\mathbf{1}_N^{T}\otimes\mathbf{\bar{u}}))^k\x(0)\|
	\end{align*}
\end{lemma}
\begin{proof} 
	\begin{align*}
	\x(k) - (\mathbf{1}_{N}^{T}\otimes\mathbf{\bar{u}})\x(k) = (\mathbf{A}\otimes \mathbf{I}_n)\x(k-1) - \\
	\alpha_{k} (\mathbf{1}_{N}^{T}\otimes \mathbf{\bar{u}}) 
	[\nabla f_i(k)]_{1\leq i \leq N} \\
	-(\mathbf{1}_{N}^{T}\otimes\mathbf{\bar{u}})\x(k-1) + \alpha_{k} (\mathbf{1}_{N}^{T}\otimes \mathbf{\bar{u}})[\nabla f_i(k)]_{1\leq i \leq N} \\	
	= (\mathbf{A}\otimes \mathbf{I}_n - \mathbf{1}_{N}^{T}\otimes \mathbf{\bar{u}})\x(k-1)
	\end{align*}
	
	\begin{flushleft}
		\text{Repeating this equation for $k-1$ gives,} 
	\end{flushleft}
	\begin{align*}
	\x(k-1) = (\mathbf{A}\otimes \mathbf{I}_n - \mathbf{1}_{N}^{T}\otimes \mathbf{\bar{u}})\x(k-2) + \hspace{15mm}\\ (\mathbf{1}_{N}^{T}\otimes \mathbf{\bar{u}})\x(k-1)
	\end{align*}
	
	\begin{flushleft}
		\text{Substituting for $\x(k-1)$ above gives,} 
	\end{flushleft}
	\begin{align*}
	\x(k) - (\mathbf{1}_{N}^{T}\otimes\mathbf{u})\x(k) = (\mathbf{A}\otimes \mathbf{I}_n - \mathbf{1}_{N}^{T}\otimes \mathbf{\bar{u}}) \hspace{15mm} \\((\mathbf{A}\otimes \mathbf{I}_n - \mathbf{1}_{N}^{T}\otimes \mathbf{\bar{u}})\x(k-2) + (\mathbf{1}_{N}^{T}\otimes \mathbf{u})\x(k-1))\\
	= (\mathbf{A}\otimes \mathbf{I}_n - \mathbf{1}_{N}^{T}\otimes \mathbf{\bar{u}})^2\x(k-2)
	\end{align*}
	
	\begin{flushleft}
		\text{Continuing this way,} 
	\end{flushleft}
	\begin{align*}
	\x(k) - (\mathbf{1}_{N}^{T}\otimes\mathbf{\bar{u}})\x(k) = (\mathbf{A}\otimes \mathbf{I}_n - \mathbf{1}_{N}^{T}\otimes \mathbf{\bar{u}})^k\x(0) 
	\end{align*}
	\begin{flushleft}
		\text{which immediately gives the result in the lemma.} 
	\end{flushleft}
 \end{proof}
 	
	\begin{flushleft}
		\text{From Lemma 3, } 
	\end{flushleft}
	\begin{align*}
	\|\x(k) - (\mathbf{1}_{N}^{T}\otimes\mathbf{\bar{u}})\x(k)\|_2 = \|(\mathbf{A}\otimes \mathbf{I}_n - \mathbf{1}_{N}^{T}\otimes \mathbf{\bar{u}})^k\x(0)\|_2 \\
	< \|(\mathbf{A}\otimes \mathbf{I}_n - \mathbf{1}_{N}^{T}\otimes \mathbf{\bar{u}})^k\|\|\x(0)\|_2 \\
	= (\rho(\mathbf{A}\otimes \mathbf{I}_n - \mathbf{1}_{N}^{T}\otimes \mathbf{\bar{u}}))^{k}\|\x(0)\|_2
	\end{align*}
	\begin{flushleft}
		\text{It is known that, }\cite{horn2012matrix} 
	\end{flushleft}
	\begin{align*}
	\rho (\mathbf{A}\otimes \mathbf{I}_n - \mathbf{1}_{N}^{T}\otimes \mathbf{\bar{u}}) < 1. 
	\end{align*}
	\begin{flushleft}
		\text{Hence for any given $\varepsilon$ we have }$k_0$ {so that,}	
	\end{flushleft}  
	\begin{align*}
	\|\x(k) - (\mathbf{1}_{N}^{T}\otimes\mathbf{\bar{u}})\x(k)\|_2 < \varepsilon 
	\end{align*}
	\begin{flushleft}
		\text{for all $k$ $>$ $k_0$}. 
	\end{flushleft} 

Define $\hat{\p}(k)$ = $ 
\displaystyle\sum_{i=1}^{N} \frac{\mathbf{u}_i\nabla f_i(\mathbf{\bar{u}}\x(k))}{[\e_{i}(k)]_{i}}$ 
and $\p(k) = [\p_1(k) \hdots \p_N(k)]^T$. From Lemma~\ref{lem4} we have,  $k_0$ such that $\forall$ $k > k_0$  
\begin{align*} \label{15}
\|\mathbf{1}_{N}^{T} \otimes \hat{\p}(k) - \p(k)\|_2 \hspace{55mm} \\ \leq  
\sum_{i=1}^{N}\|\frac{\mathbf{u}_i\nabla f_i(\mathbf{\bar{u}}\x(k))}{[\e_{i}(k)]_{i}}- \frac{\mathbf{u}_i\nabla f_i(\x_i(k))}{[\e_{i}(k)]_{i}}\|_2 
\\ = \sum_{i=1}^{N}\frac{\mathbf{u}_i}{[\e_{i}(k)]_{i}}\|\nabla f_i(\mathbf{\bar{u}}\x(k))-\nabla f_i(\x_i(k))\| \nonumber \\
< \sum_{i=1}^{N}\frac{\mathbf{u}_il_i}{[\e_{i}(k)]_{i}}\|\mathbf{\bar{u}}\x(k)-\x_i(k)\|
< L_h\epsilon \sum_{i=1}^{N}\frac{\mathbf{u}_i}{[\e_{i}(k)]_{i}} \approx NL_h\varepsilon \nonumber \\ \text{for "large" $k$.} \tag{19}
\end{align*} 

We now consolidate the results in Lemma 2 and Lemma 3 to present a contraction result on the distance between the weighted average of local estimates of global optimum and the global optimum, $\x^{*}$ in Theorem 1.  
\begin{theorem}\label{thm:errorinstrconv}
	Let Assumptions~\ref{asenv1}-\ref{asenv3} hold. Let $\x(k)$ = $[\x_i(k)]_{1\leq i \leq N}^{T}$  be the sequence of estimates generated by 14(e). Let $L_h$ = $\max_{1\leq i \leq N}l_i$ and $\mu_{1i}$, $\mu_{2i}$ be the corresponding constants for $f_{i}$ as defined in Lemma 2.  
	Then $\exists$ $k_1$ such that $\forall k \geq k_1$ we have,
	\begin{align*}
	\hspace{-0.1in} \big \|\bar{\mathbf{u}}\x(k+1) - \x^* \big\|^{2}  \leq
	& a_{1}(k) \big \|\bar{\mathbf{u}}\x(k) - \x^* \big \|^{2} + O(\varepsilon).
	\end{align*}
	where $a_{1}(k)$ = $4 + 4N\displaystyle\sum_{i=1}^{N} \bigg(\frac{l_{i}\alpha_{k}\mathbf{u}_i}{[\e_{i}(k)]_{i}}\bigg)^{2}-\frac{8\alpha_{k}\mathbf{u}_{i}l_i}{[\e_{i}(k)]_{i}} $
\end{theorem}

\begin{proof}
	\begin{align*} \label{16}
	\|\bar{\mathbf{u}}{\x}(k+1) - \x^*\|^2 = \|\bar{\mathbf{u}}{\x}(k) -  \bar{\mathbf{u}}\alpha_{k} [\nabla f_i(k)]_{1\leq i\leq N} - \x^*\|^2  \\
	= \|\bar{\mathbf{u}}{\x}(k) -  \bar{\mathbf{u}}\alpha_{k} [\nabla f_i(k)]_{1\leq i\leq N} +\alpha_{k}{\hat{\p}(k)} - \alpha_{k}{\hat{\p}(k)} - \x^*\|^2 \\
	\leq 2\|\bar{\mathbf{u}}\x(k) - \x^{*} -\alpha_{k}\hat{\p}(k) \|^2 \hspace{40mm} \\ +2\alpha_{k}^{2} \|\sum_{i=1}^{N} \frac{\mathbf{u}_i\nabla f_i(\x_{i}(k))}{[\e_{i}(k)]_{i}} - {\hat{\p}(k)}\|^2 \tag{20}
	\end{align*}
	\begin{flushleft}
		From (19) there is $k_0$ so that $\forall k > k_0$ the second term is bounded by $\varepsilon$. The first term in $\eqref{16}$ can be written as, 
	\end{flushleft}
	\begin{align*} \label{17}
	2\|(\bar{\mathbf{u}}\x(k) - \x^*) -\alpha_{k} \sum_{i=1}^{N} \frac{\mathbf{u}_i}{[\e_{i}(k)]_{i}} (\nabla f_i(\mathbf{\bar{u}}\x(k))-\nabla f_i(\x^{*})) - \\ \sum_{i=1}^{N} \frac{\alpha_{k} \mathbf{u}_i}{[\e_{i}(k)]_{i}} \nabla f_i(\x^{*})\|^2  \tag{21}
	\end{align*}
	\vspace{-6mm}
	\begin{flushleft}
	From optimality condition, $\displaystyle\sum_{i=1}^{N}\nabla f_i(\x^{*}) = 0.$ Hence (21) can be re-written as,
	\end{flushleft} 
    \vspace{-2mm}
	\begin{align*} \label{18}
	 2\|(\bar{\mathbf{u}}\x(k) - \x^*) -\alpha_{k} \sum_{i=1}^{N} \frac{\mathbf{u}_i}{[\e_{i}(k)]_{i}} (\nabla f_i(\mathbf{\bar{u}}\x(k))-\nabla f_i(\x^{*})) - \\ \sum_{i=1}^{N} \frac{\alpha_{k} \mathbf{u}_i}{[\e_{i}(k)]_{i}} \nabla f_i(\x^{*}) + \alpha_{k} \sum_{i=1}^{N}\nabla f_i(\x^{*}) \|^2 \tag{22}
	\end{align*}
	=
	\vspace{-3mm}
	\begin{align*} \label{19}
		2\|(\bar{\mathbf{u}}\x(k) - \x^*) -\alpha_{k} \sum_{i=1}^{N} \frac{\mathbf{u}_i}{[\e_{i}(k)]_{i}} (\nabla f_i(\mathbf{\bar{u}}\x(k))-\nabla f_i(\x^{*})) - \\ \sum_{i=1}^{N} \alpha_{k} \nabla f_i(\x^{*}) \bigg(1- \frac{\mathbf{u}_i}{[\e_{i}(k)]_{i}}\bigg)\|^{2} \tag{23}
	\end{align*}
	$\leq$
	\begin{align*} \label{20}
	4\|(\bar{\mathbf{u}}\x(k) - \x^*) +\alpha_{k} \sum_{i=1}^{N} \frac{\mathbf{u}_i}{[\e_{i}(k)]_{i}} (\nabla f_i(\x^*)-\nabla f_i(\mathbf{\bar{u}}\x(k)))\|^2  \\+ 4\|\sum_{i=1}^{N} \alpha_{k}\nabla f_i(\x^{*}) \bigg(1- \frac{\mathbf{u}_i}{[\e_{i}(k)]_{i}}\bigg)\|^2 \tag{24}
	\end{align*}
	\begin{flushleft}
		The first term in $\eqref{20}$ can be bounded as, 
	\end{flushleft}
	\begin{align*} \label{22}
	4\|(\bar{\mathbf{u}}\x(k) - \x^*) -\alpha_{k} \sum_{i=1}^{N} \frac{\mathbf{u}_i}{[\e_{i}(k)]_{i}} (\nabla f_i(\mathbf{\bar{u}}\x(k))-\nabla f_i(\x^{*}))\|^2 
	\hspace{30mm} \end{align*} 
	=
	\begin{align*}
	4\|\sum_{i=1}^{N}\frac{1}{N}(\bar{\mathbf{u}}\x(k) - \x^*) - \frac{\alpha_{k}\mathbf{u}_i}{[\e_{i}(k)]_{i}} (\nabla f_i(\mathbf{\bar{u}}\x(k))-\nabla f_i(\x^{*}))\|^2 \hspace{34mm} \end{align*}
	$\leq$
	\begin{align*}
	 4N\sum_{i=1}^{N}\|\frac{(\bar{\mathbf{u}}\x(k) - \x^*)}{N} - \frac{\alpha_{k}\mathbf{u}_i}{[\e_{i}(k)]_{i}} (\nabla f_i(\mathbf{\bar{u}}\x(k))-\nabla f_i(\x^{*}))\|^2 \hspace{31mm} \end{align*}
	 =
	\begin{align*}
	4N\sum_{i=1}^{N}\frac{1}{N^2}\|\bar{\mathbf{u}}\x(k) - \x^*\|^{2} - \frac{2\alpha_{k}\mathbf{u}_i}{N[\e_{i}(k)]_{i}}\langle \bar{\mathbf{u}}\x(k) - \x^*, \hspace{45mm} \\ \nabla f_i(\mathbf{\bar{u}}\x(k))-\nabla f_i(\x^{*})\rangle \hspace{40mm} \\ + 
	\bigg(\frac{\alpha_{k}\mathbf{u}_i}{[\e_{i}(k)]_{i}}\bigg)^{2} \|\nabla f_i(\mathbf{\bar{u}}\x(k))-\nabla f_i(\x^{*})\|^{2}\hspace{60mm} 
	\end{align*}
	$\leq$
	\begin{align*}
	 4N\sum_{i=1}^{N}\frac{1}{N^2} + \bigg(\frac{\alpha_{k}\mathbf{u}_i l_i}{[\e_{i}(k)]_{i}}\bigg)^{2}-  \frac{2\alpha_{k}\mathbf{u}_i}{N[\e_{i}(k)]_{i}}(\mu_{2i}\|\bar{\mathbf{u}}\x(k) - \x^*\|^2 +  \\ \mu_{1i}\|\nabla f_i(\mathbf{\bar{u}}\x(k))-\nabla f_i(\x^{*})\|^2) \end{align*} = \begin{align*}
	  \bigg(\sum_{i=1}^{N}\frac{4}{N} + 4N\bigg(\frac{\alpha_{k}\mathbf{u}_i l_i}{[\e_{i}(k)]_{i}}\bigg)^{2} - \frac{8\alpha_{k}\mathbf{u}_i}{[\e_{i}(k)]_{i}}(\mu_{1i}l_i^{2}+\mu_{2i})\bigg) \hspace{10mm}\\
	\|\bar{\mathbf{u}}\x(k) - \x^*\|^{2} \tag{25}\\
	 = \bigg(\sum_{i=1}^{N}\frac{4}{N} + 4N\bigg(\frac{\alpha_{k}\mathbf{u}_i l_i}{[\e_{i}(k)]_{i}}\bigg)^{2} - \frac{8\alpha_{k}\mathbf{u}_i}{[\e_{i}(k)]_{i}}l_i \bigg)
	\|\bar{\mathbf{u}}\x(k) - \x^*\|^{2} \tag{26}
	\end{align*}
	\begin{flushleft}
		where (25) is obtained by invoking Lemma 2 and (26) is a consequence of smoothness of $f_i$. 
	\end{flushleft}	
	\begin{flushleft}
		 The second term in $\eqref{20}$ is bounded by, \cite{mai2016distributed}
	\end{flushleft}
	\begin{align*} \label{23}
	 4\left\Vert \sum_{i=1}^{N} \alpha_{k}\nabla f_i(\x^{*}) \bigg(1- \frac{\mathbf{u}_i}{[\e_{i}(k)]_{i}}\bigg) \right\Vert^2 \leq C\lambda^{k} \sum_{i=1}^{N} \frac{\nabla f_i(\x^{*})}{[\e_{i}(k)]_{i}}  \tag{27}
	\end{align*}
	
	\begin{flushleft}
		where $C>0$ is a constant and $0 <\lambda< 1$.
	\end{flushleft}
	\begin{flushleft}
	Thus, there is $k_2$ so that $\forall k > k_2$ the R.H.S. in $\eqref{23}$ is less than $\varepsilon$. Choosing $k_1$ = $\max\{k_{0},k_{2}\}$ and combining $\eqref{16}$, $\eqref{20}$, $\eqref{22}$, $\eqref{23}$ we get the result.
    \end{flushleft}
\end{proof}

Before we give a convergence proof for the proposed algorithm in (18) we state an additional assumption on the smoothness parameters of the nodal functions. The assumption shall aid in our proof as will be evident in the subsequent analysis.   
\begin{assumption}\label{asenv5}
	For each $i$ we have, $\frac{\big(\displaystyle\sum l_{i}\big)^{2}}{N\displaystyle\sum l_{i}^{2}} > \frac{3}{4}$. 
\end{assumption}
We now present a proof of the main result of the paper. The proof proceeds by imposing a condition on the contraction result obtained in Theorem 1 to obtain a viable step size for which convergence within an $\epsilon$-ball around the global optimum is guaranteed. 

\begin{theorem} \label{Theorem3}
	Let Assumptions~\ref{asenv1}-\ref{asenv3} and \ref{asenv5} hold. Let $\{x_{i}(k)\}_{1\leq i\leq N}$ be the sequence generated in (13d). Then $\forall i$ there is $k_3$ such that for each k $\geq$ $k_3$,
	\begin{align*}
	\|\x_{i}(k) - \x^{*}\| < O(\varepsilon) 
	\end{align*}  
	\begin{flushleft}
		under appropriate choice of step size $\alpha_{k}$.
	\end{flushleft}
\end{theorem}
\begin{proof}
	\begin{flushleft}
		We must choose $\alpha_{k}$ so that $|a_{i}(k)| < 1$ for each $k>0$. Observe that $a_{i}(k)$ is a quadratic in $\alpha_k$. Thus, rearranging $a_{1}(k)$ in Theorem 1 it is desired that, 
	\end{flushleft}
	\begin{align} \label{24}
	0 < A_{k}\alpha_k^{2} + B_{k}\alpha_{k} + 4 < 1 \tag{28}
	\end{align}
	\begin{flushleft}
	where, \\ $A_{k} = 4\displaystyle\sum_{i=1}^{N}Nl_{i}^{2}\bigg(\frac{\mathbf{u}_{i}}{[\e_{i}(k)]_{i}}\bigg)^{2}$ , $B_{k} = \displaystyle\sum_{i=1}^{N}-8\frac{\mathbf{u}_il_i}{[\e_{i}(k)]_{i}}$ 
\end{flushleft}
\begin{flushleft}
	The minimum value of the expression above is $4 -\frac{B_{k}^{2}}{4A_{k}}$, which is obtained at $\alpha_{k} = \frac{-B_{k}}{2A_{k}}$. By Assumption \ref{asenv5}, we have, $\frac{\big(\displaystyle\sum l_{i}\big)^{2}}{N\displaystyle\sum l_{i}^{2}} > \frac{3}{4}$ and from Cauchy-Schwarz Inequality we have, $\frac{\big(\displaystyle\sum l_{i}\big)^{2}}{N\displaystyle\sum l_{i}^{2}} < 1$. Thus,
	$ 0 < 4 - \frac{64\big(\displaystyle\sum l_{i}\big)^{2}}{16N\displaystyle\sum l_{i}^{2}}$ = $4 -\frac{B_{k}^{2}}{4A_{k}} < 1.$	\\ 
\end{flushleft}
\begin{flushleft}
	Reproducing the relation from Theorem~\ref{thm:errorinstrconv} for $k = 0$, we have,
\end{flushleft}
	\begin{align*}
	 \big \|\bar{\mathbf{u}}\x(1) - \x^* \big\|^{2}  \leq
		& a_{1}(0) \big \|\bar{\mathbf{u}}\x(0) - \x^* \big \|^{2} + O(\varepsilon).	\tag{29}
	\end{align*}
	\begin{flushleft}
		and for $k=1$, 
	\end{flushleft}
	\vspace{-3mm}
	\begin{align*}
	\big \|\bar{\mathbf{u}}\x(2) - \x^* \big\|^{2}  \leq
	& a_{1}(1) \big \|\bar{\mathbf{u}}\x(1) - \x^* \big \|^{2} + O(\varepsilon).	\tag{30}
	\end{align*}
	\begin{flushleft}
		Substituting (29) in (30) gives, 
	\end{flushleft}
	\vspace{-3mm}
	\begin{align*}
	\big \|\bar{\mathbf{u}}\x(2) - \x^* \big\|^{2}  \leq
	& a_{1}(1)a_{1}(0) \big \|\bar{\mathbf{u}}\x(0) - \x^* \big \|^{2} + O(\varepsilon).	\tag{31}
	\end{align*}
	\begin{flushleft}
		Generalizing, 
	\end{flushleft}
	\vspace{-5mm}
	\begin{align*}	
	\big \|\bar{\mathbf{u}}\x(k) - \x^* \big\|^{2}  \leq
	& \prod_{s=0}^{k-1}a_{1}(s)\big \|\bar{\mathbf{u}}\x(0) - \x^* \big \|^{2} + O(\varepsilon).	\tag{32}
	\end{align*}
	\begin{flushleft}
	Since $a_{1}(s) < 1$ for each $s$ by our choice of $\alpha_{s}$, $\exists$ $k_{3}$ such that $\forall$ $k > k_{3}$ $\prod_{s=0}^{k-1}a_{1}(s) < \varepsilon$, which immediately gives the result in the theorem.     	
	\end{flushleft}
\end{proof}

\section{Simulations}
For the purpose of simulations we analyse two different cases. It is easy to verify that the functions assigned to nodes in both the scenarios are consistent with the assumptions \ref{asenv1}-\ref{asenv4}. In the first case $f_i(x)$ is assigned to be $(x+i)^2$ $\forall$ $i$. In this case the local minimas are regularly spaced and the global minima coincides with the average of local minimas. We observe that in accordance with our theoretical analysis, convergence to global minima for algorithm (18) is observed in the figures 1 and 2. 
\begin{figure}[h]
	\centering
	\includegraphics[width=0.5\textwidth]{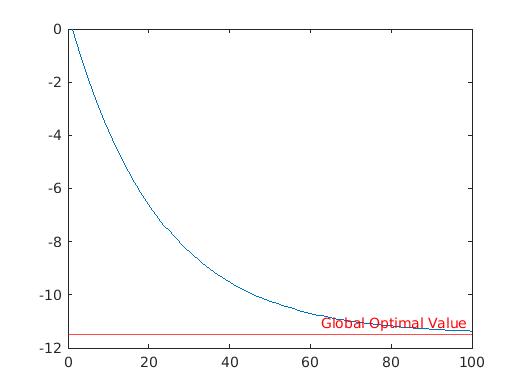}
	\caption{Convergence of agent 1 to global minima located at -11.5. A total of 22 agents are connected through a directed graph topology.}
\end{figure}
\begin{figure}[h]
	\centering
	\includegraphics[width=0.5\textwidth]{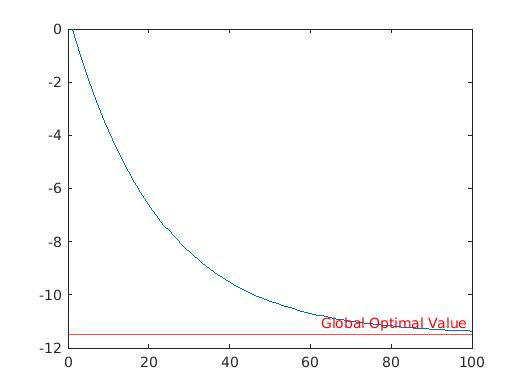}
	\caption{Convergence of agent 10 to global minima located at -11.5.}
\end{figure}

In the second scenario we force the minima of the first local function to become an outlier. Hence our choice of local functions is as follows: \\
\[
f_i(x)= 
\begin{cases}
(x+i)^2,& \text{if } i\neq 1\\
(x+100)^2,              & \text{otherwise}
\end{cases}
\] 
The convergence to global minima in this case for agents 1 and 10 can be observed in Figures 3 and 4 respectively.  
\begin{figure}[h]
	\centering
	\includegraphics[width=0.5\textwidth]{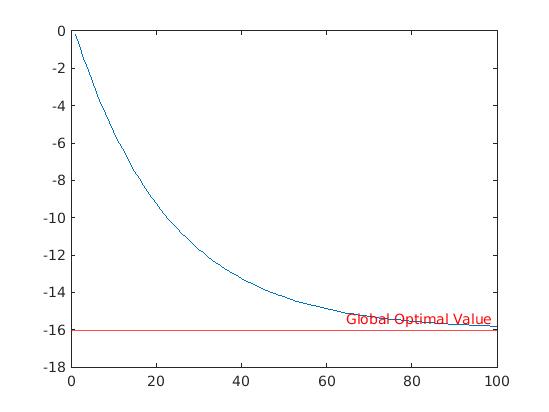}
	\caption{Convergence of agent 1 to global minima located at -16.045 approx.}
\end{figure}
\begin{figure}[h]
	\centering
	\includegraphics[width=0.5\textwidth]{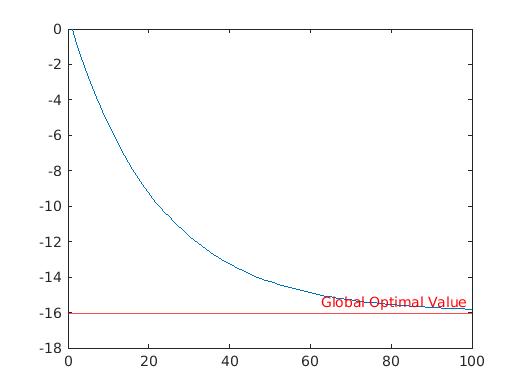}
	\caption{Convergence of agent 10 to global minima located at -16.045 approx.}
\end{figure}
We demonstrate the utility of the iterates in (18a-e) by comparing the average tracker in (18b) and (18c) with a naive scheme where the gradients are being communicated by a node to it's out neighbours, followed by a direct averaging. In our system of 22 nodes, the nodes have been inititialised as $r_{i}(1)= \frac{\nabla f_{i}(\x_{i}(0))}{e_{ii}(0)}$. The nodes are to track the average of these quantities weighted by elements of the first left eigenvector of the adjacency matrix at every $\frac{T_g}{2}$ time steps. The maximum allowed delay between nodes is 157 units which corresponds to the choice $\kappa = 0.01$. Under naive averaging each node will have to wait for 157 x 21 = 3297 time units in the worst case for obtaining the local gradient information from all the nodes. However, under average tracking we see that the nodes approach the weighted average well within 3000 time units. This is summarised in Fig. 5. Additionally, we compare the nature of convergence to the optimum with average gradient tracking and with naive averaging in Fig. 6.    
\begin{figure}[h]
	\centering
	\includegraphics[width=0.5\textwidth]{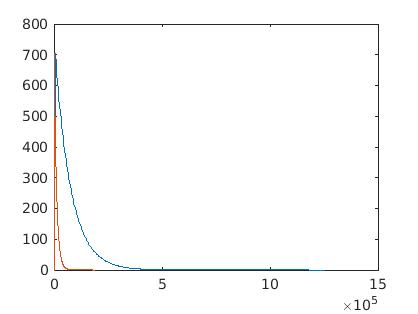}
	\caption{The red graph shows the iterates of the average tracker converging to the average of gradients for Node 1. The blue graph depicts the convergence for the naive scheme.}
\end{figure}

\begin{figure}[h]
	\centering
	\includegraphics[width=0.5\textwidth]{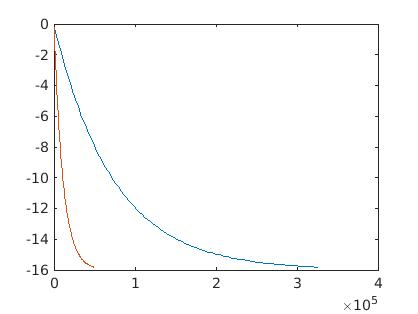}
	\caption{The red graph shows the iterates in the average tracker case converging to the global optimum for Node 1. The blue graph depicts the convergence for the naive scheme.}
\end{figure}
\section{Conclusion}
In this paper, we consider distributed optimization for graphs with row-stochastic weights. Most of the existing algorithms are based on synchronised communication among agents, which may be infeasible to implement in many practical scenarios since communication delays are inevitable. We propose an algorithm inspired by FROST and show that it converges under appropriate choice of step sizes. Our algorithm inherits the convergence properties of FROST while being robust to time varying communication delays. Simulation results substantiate our theoretical claims. Several possibilities of future work include replicating our algorithm in an adversarial setting while ensuring convergence to the optimum.

\bibliographystyle{ieeetran}
\bibliography{scibib} 

\section{APPENDIX A}
\begin{proof}
	\begin{flushleft}
		Define, 
	\end{flushleft}
	\begin{align*}
      g(\x) = F(\x) - \frac{\sigma_{f}}{2}\|\x\|^2
	\end{align*}
	\begin{flushleft}
		Strong convexity of $F(\x)$ implies that $g(\x)$ is convex. Consider the function,
	\end{flushleft}
	\begin{align*}
		h(\x) = \frac{L_{f} - \sigma_{f}}{2}\|\x\|^2 - g(\x) = \frac{L_f}{2}\|\x\|^2 - F(\x)
	\end{align*}
	\begin{flushleft}
		Smoothness of $F(\x)$ implies that $h(\x)$ is convex which inturn implies that $g(\x)$ is smooth with parameter $L_{f} - \sigma_{f}$. Applying co-coercivity property for smooth functions on $g(\x)$ gives,
	\end{flushleft}
	\begin{align*}
    	\langle \nabla g(\x) - \nabla g(\y), \x - \y \rangle \geq \frac{1}{L_f - \sigma_{f}}\|\nabla g(\x) - \nabla g(\y)\|^2 \hspace{15mm} \\
	 \implies \langle \nabla F(\x) - \nabla F(\y) - \sigma_{f}(\x - \y), \x - \y \rangle \hspace{32mm} \\ \geq \frac{1}{L_f - \sigma_{f}}\|\nabla F(\x) - \nabla F(\y) - \sigma_{f}(\x - \y)\|^2 \hspace{15mm} \\
	 \end{align*}
	 \vspace{-10mm}
	 \begin{align*}
	\implies (1 + \frac{2\sigma_{f}}{L_f - \sigma_f}) \langle \nabla F(\x) - \nabla F(\y), \x - \y \rangle \hspace{37mm} \\  \geq (\frac{1}{L_f - \sigma_f})\|\nabla F(\x) - \nabla F(\y)\|^2 + (\frac{\sigma_{f}^{2}}{L_f - \sigma_{f}} + \sigma_{f})\|\x - \y\|^2 \hspace{13mm}
	\end{align*}
	\begin{flushleft}
	which gives the result in Lemma 2.
	\end{flushleft}
\end{proof} 

\end{document}